\documentclass[10pt, english]{amsart}

\usepackage{amsmath,amssymb,enumerate}

\usepackage[T1]{fontenc}
\usepackage[all,cmtip]{xy}

\usepackage{babel}
\usepackage{amstext}
\usepackage{amsmath}
\usepackage{amsfonts}
\usepackage{latexsym}
\usepackage{ifthen}

\usepackage{xypic}
\xyoption{all}
\makeatletter
\@namedef{subjclassname@2020}{\textup{2020} Mathematics Subject Classification}
\makeatother

\newcommand{\Bs}{{\rm Bs}}

\newtheorem{lemma1}{}[section]

\newenvironment{lemma}{\begin{lemma1}{\bf Lemma.}}{\end{lemma1}}
\newenvironment{example}{\begin{lemma1}{\bf Example.}\rm}{\end{lemma1}}

\newenvironment{theorem}{\begin{lemma1}{\bf Theorem.}}{\end{lemma1}}

\newenvironment{corollary}{\begin{lemma1}{\bf Corollary.}}{\end{lemma1}}

\newenvironment{definition}{\begin{lemma1}{\bf Definition.}}{\end{lemma1}}

\newenvironment{setup}{\begin{lemma1}{\bf Setup.}}{\end{lemma1}}

\newenvironment{conjecture}{\begin {lemma1}{\bf Conjecture.}}{\end{lemma1}}

\newenvironment{fact}{\begin{lemma1}{\bf Fact.}}{\end{lemma1}}

\newenvironment{remark*}{{\bf Remark.}}{}
\newenvironment{remarks*}{{\bf Remarks.}}{}
\newenvironment{example*}{{\bf Example.}}{}
\newenvironment{assumption*}{{\bf Assumption.}}{}

\newcommand{\R}{\ensuremath{\mathbb{R}}}
\newcommand{\Q}{\ensuremath{\mathbb{Q}}}
\newcommand{\Z}{\ensuremath{\mathbb{Z}}}

\newcommand{\N}{\ensuremath{\mathbb{N}}}
\newcommand{\PP}{\ensuremath{\mathbb{P}}}

\usepackage{eurosym}

\newcommand{\holom}[3]{\ensuremath{#1:#2  \rightarrow #3}}
\newcommand{\fibre}[2]{\ensuremath{#1^{-1} (#2)}}

\makeatletter
\ifnum\@ptsize=0  \fi \ifnum\@ptsize=2  \fi 

\newcommand\sO{{\mathcal O}}

\DeclareMathOperator*{\pic}{Pic}

\usepackage{xcolor}
\usepackage{hyperref}

\author{Daniele Agostini}
\author{Andreas H\"oring}

\address{Daniele Agostini, Fachbereich Mathematik, University of T\"ubingen, Auf der Morgenstelle 10, 72072 T\"ubingen, Germany}
\email{daniele.agostini@uni-tuebingen.de}

\address{Andreas H\"oring, Universit\'e C\^ote d'Azur, CNRS, LJAD, France}
\email{Andreas.Hoering@univ-cotedazur.fr}


\DeclareMathOperator*{\BKX}{\ensuremath{\overline{\mathcal{BK}}_X}}
\DeclareMathOperator*{\CX}{\ensuremath{\overline{\mathcal{C}_X}}}

\title{Fixed divisors on hyperk\"ahler manifolds} 
\date{\today} 

\begin{document}

\begin{abstract} 
Let $X$ be a hyperk\"ahler manifold, and let $A$ be a nef and big divisor on $X$.
We show that the fixed part of the linear system $|A|$ is reduced and as a consequence
$|2A|$ is mobile. If $X$ has dimension four we also show that if the fixed part of $|A|$ is not empty, the mobile part induces a (rational) Lagrangian fibration. 
\end{abstract} 

\maketitle

\section{Introduction}

\subsection{Motivation and main results}

Given an ample divisor on a complex projective manifold with trivial canonical class, one of the most fundamental questions is the existence of global sections and the geometry of the base locus.
For abelian varieties this is a classical subject \cite[Section 4]{BL04}, for K3 surfaces we have
a complete answer due to Mayer \cite{May72} (cf. \cite{SD74} for a proof in positive characteristic):

\begin{theorem} \label{theorem:mayer}
Let $X$ be a smooth K3 surface, and let $A$ be a nef and big divisor on $X$ such that the base locus $\Bs(|A|)$ is not empty. Then we have a decomposition into mobile and fixed part
$$
A \simeq d L + B
$$
where $d \geq 2$ and $|L|$ defines an elliptic fibration 
$$
\holom{\varphi = \varphi_{|L|}}{X}{\PP^1}
$$
such that $B \subset X$ is a section. In particular $B$ is a $(-2)$-curve and $\Bs(|A|)=B$.
\end{theorem}

A surprising feature of this result is that the existence of base points implies that
there is a fixed divisor, i.e. the base locus has a component of codimension one.
This is not true for $K$-trivial manifolds of higher dimension (cf. \cite{Rie20} for manifolds
of $\mbox{K3}^{[2]}$-type), so an interesting first step is to describe the geometry assuming the existence of a fixed divisor. For Calabi-Yau threefolds this is essentially uncharted territory (cf. \cite[App.2]{BN16} or \cite[Thm.1.4]{HS25} for partial results), but the case of hyperk\"ahler manifolds (aka irreducible holomorphic symplectic manifolds) seems more accessible:
in \cite{Rie23} Ulrike Rie\ss\ showed that the fixed part of the linear system
is a prime divisor with negative Beauville-Bogomolov square if the manifold satisfies certain conjectural properties including a strong form of the SYZ conjecture. In this paper we show a slightly weaker version that does not require
conjectural assumptions:

\begin{theorem} \label{theorem:main}
Let $X$ be a hyperk\"ahler manifold, and let $A$ be a nef and big divisor on $X$
such that the linear system $|A|$ has a fixed divisor. Then the following holds:
\begin{itemize}
\item The fixed part of $|A|$ is a reduced divisor\footnote{We give  a precise description of
the fixed part in Theorem \ref{theorem-technical}.}.
\item The linear system $|2A|$ has no fixed part.
\end{itemize}
If $A$ is ample we have a decomposition into mobile and fixed part
$$
A \simeq M + B
$$
where $q(M)=0$ and $B$ is a prime divisor with $q(B) \leq 0$.
\end{theorem}

Compared to Mayer's theorem \ref{theorem:mayer} this result lacks two features:
up to performing a sequence of flops the mobile part $M$ should define a Lagrangian
fibration onto the projective space, moreover the fixed part should have negative square and therefore be uniruled \cite[Prop.4.7]{Bou04}.
In dimension four we obtain the stronger result: 

\begin{theorem} \label{thm:HK4}
Let $X$ be a smooth hyperk\"ahler fourfold, and let $A$ be a nef and big divisor on $X$ such that the linear system $|A|$ has a fixed divisor. Then we have a decomposition into mobile and fixed part
$$
A \simeq dL+B 
$$ 
where $d \geq 2$ and $B$ is a prime divisor with $q(B)<0$.  
Moreover, up to replacing $X$ by a birational model, the linear system $|L|$ is basepoint free and defines a Lagrangian fibration
$\holom{\varphi}{X}{\PP^2}$ such that $\sO_X(L) \simeq \varphi^* \sO_{\PP^2}(1)$.
\end{theorem}

Note that the hyperk\"ahler SYZ conjecture (which is part of the assumptions in \cite{Rie23}) is not known for hyperk\"ahler fourfolds. A key contribution of this paper is to develop arguments that avoid appealing to this very difficult open problem. The following classical example shows that the situation described in 
Theorem \ref{thm:HK4} exists:

\begin{example} \label{example:Beauville-Mukai} \cite[Example 0.5]{Muk84} \cite{Bea91}
Let $S$ be a smooth K3 surface with $\pic{S} \simeq \Z H_S$ where $H_S$ is an ample divisor
and $H_S^2=2g-2$. The linear system $|H_S|$ has dimension $g$ and 
the universal family of Jacobians $\mbox{\rm Jac}(C_t)$ for $C_t \in |H_S|$ general compactifies
to a hyperk\"ahler manifold $X$ of dimension $2g$ with a Lagrangian fibration
$$
\holom{f}{X}{|H_S| \simeq \PP^g}.
$$
The symmetric powers $C_t^{(g-1)}$ determine a divisor $B \subset X$ such that
the restriction of $B$ to a general fibre  $\mbox{\rm Jac}(C_t)$ is a principal polarisation.
Thus for all $d \in \N$ the divisor $B$ is contained
in the fixed part of
$$
A := f^* dH + B
$$
where $H$ is the hyperplane class on $\PP^g$. It is known that $q(B)=-2$ and
$f^* dH + B$ is nef and big for $d \geq 2$ (cf. \cite{Rie20} for a very detailed exposition
of the case $g=2$).
\end{example}

A consequence of Mayer's theorem \ref{theorem:mayer} is that the linear system $|2A|$ is always basepoint free, in higher dimension this is not clear.
Basepoint freeness was studied by Rie\ss \ \cite{Rie20} and Varesco \cite{Var23} for low-dimensional hyperk\"ahler manifolds of $K3^{[n]}$ and Kummer type, but their results
only apply to general members in the moduli space. 

\subsection{Strategy of proof and further research}

Let $X$ be a hyperk\"ahler manifold, and denote by $q(\bullet)$ the Beauville-Bogomolov form on $X$ (BB-form for short).
Let $A$ be a nef and big divisor such that the linear system $|A|$ has a fixed part $B \neq 0$.
The basic strategy of our proof is the one of Rie\ss' inspiring paper \cite{Rie23}, a natural generalisation of the arguments for K3 surfaces \cite[Section 2]{Huy16}: use vanishing theorems for movable divisors and properties of the Huybrechts-Riemann-Roch polynomial $P(q(\bullet))$ to obtain restrictions on the BB-square of the mobile and fixed part.
A brilliant recent paper of Chen Jiang \cite{Jia23} showing that $P$ is a strictly
increasing function allows us to compare dimensions of movable linear systems just by knowing
their BB-square. With this in mind our first idea is that we should not start the investigation by considering
$B$ as a rigid divisor (i.e. as an effective divisor with $h^0(X, \sO_X(B))=1$), but use the stronger property that $B$ is the fixed part of a movable linear system.
The second idea is that we can replace the divisibility of the mobile part by 
easy arithmetic properties of the intersection numbers provided by Markman's theory (cf. Lemma 
\ref{lemma:divisibility}).
Combined with the index theorem for the BB-form this quickly leads to numerous restrictions
on the intersection numbers (e.g. Lemma \ref{lemma:key}), modulo a certain computational effort
we then obtain a precise description of the fixed part in Theorem \ref{theorem-technical}.

Theorem \ref{theorem:main} is not optimal, in particular 
we expect that the fixed part is always a prime divisor even if $A$ is merely nef and big.
This would follow from the following

\begin{conjecture} 
Let $X$ be a hyperk\"ahler manifold of dimension $2n$, and let $M$ be a nef divisor on $X$
such that $q(M)=0$ and $h^0(X, \sO_X(M)) \geq n+2$. Then $M$ is divisible in the Picard group
$\pic{X}$.
\end{conjecture}

For the proof of Theorem \ref{thm:HK4} we first observe that the existence of a Lagrangian
fibration follows from 

\begin{conjecture} (Weak hyperk\"ahler SYZ conjecture)
Let $X$ be a (not necessarily projective) hyperk\"ahler manifold of dimension $2n$, and let $M$ be a nef divisor on $X$
such that $q(M)=0$ and $h^0(X, \sO_X(M)) \geq n+1$. Then $M$ is semiample, i.e. some positive multiple of $M$ is globally generated.
\end{conjecture}

Contrary to the classical hyperk\"ahler SYZ conjecture (e.g. \cite[Conj.1.2]{DHMV24}) this follows  in dimension four from well-established MMP techniques (\cite[Thm.1.5]{Fuk02} and \cite[Thm.1.3]{COP10}, \cite[Thm.4.1]{HLL25} for the K\"ahler case). The second ingredient is the description of the base variety of Lagrangian fibrations \cite{Hwa08, Ou19, HX22}, an open but maybe accessible problem in arbitrary dimension (cf. \cite{MX25} for recent progress). This geometric information     determines the dimensions of the linear systems and allows us to deduce the numerical properties of mobile and fixed part.

Theorem \ref{thm:HK4} shows that the presence of a fixed divisor is closely related
to the existence of a Lagrangian fibration on (a birational model of) the hyperk\"ahler manifold.
It seems likely that, as in Example \ref{example:Beauville-Mukai}, the fixed divisor $B$ restricts
to a principal polarisation on the general fibre. 

It is natural to ask for a generalisation of Theorem \ref{theorem:main} to singular irreducible holomorphic symplectic varieties, but there are at least two obstacles: in general
ample Weil $\Q$-Cartier divisor do not enjoy the same effective nonvanishing properties
as ample Cartier divisors \cite{Xie05}. Moreover Jiang's theorem \cite{Jia23} is only known in the smooth case.

\noindent\textbf{Acknowledgements.} 
We thank the participants of the POK0 workshop in Cetraro for stimulating discussion that initiated this project. We thank V. Benedetti and F. Denisi for comments on a first version of this text.

The authors were partially supported by the ANR-DFG project ``Positivity on K-trivial varieties'', ANR-23-CE40-0026 and DFG Project-ID 530132094. A.H.\ was also partially supported  by the France 2030 investment plan managed by the ANR, as part of the Initiative of Excellence of Universit\'e C\^ote d'Azur, reference ANR-15-IDEX-01.

\section{Notations and basic facts}

We work over the complex numbers, for general definitions we refer to \cite{Har77}. 
Varieties will always be supposed to be irreducible and reduced. 
We use the terminology of \cite{Deb01, KM98}  for birational geometry and notions from the minimal model program. We follow \cite{Laz04a} for algebraic notions of positivity.

A divisor is always a $\Z$-divisor,
and we denote by $\simeq$ the linear equivalence.
For intersection computations with the Beauville-Bogomolov form we identify a divisor $D$ with its Chern class $c_1(D)$.

We denote by $\N$ the natural numbers including zero, and by $\N^*$ the positive natural numbers.

\begin{definition}
Let $X$ be a projective manifold. 
\begin{itemize}
\item  A divisor $D$ is mobile if the base scheme $\Bs(|D|)$ has codimension at least two in $X$. If $D$ is not mobile, we call fixed part of $|D|$ the largest effective divisor $B$ contained in
$\Bs(|D|)$.
\item An effective divisor $B$ is fixed if there exists a nef and big divisor $A$ such that $B$ is contained in the fixed part
of $|A|$.
\end{itemize}
\end{definition}

\begin{lemma} \label{lemma:strict-increase}
Let $X$ be a projective manifold and $M$ a divisor on $X$ such that $h^0(X, \sO_X(M)) \geq 2$
Then $h^0(X, \sO_X(2M))> 2h^0(X,\sO_X(M))-2\geq h^0(X, \sO_X(M))$.
\end{lemma}

\begin{proof}
The multiplication map 
$$
H^0(X,\sO_X(M)) \otimes H^0(X,\sO_X(M)) \to H^0(X,\sO_X(2M))
$$ 
is injective separately on each factor, so that the result follows from the Hopf inequality 
\cite[p.108]{ACGH85}.
\end{proof}

For the rest of this section 
let $X$ be a hyperk\"ahler manifold of dimension $2n$, and denote by 
$$
q : H^2(X, \Z) \rightarrow \Z
$$
the Beauville-Bogomolov quadratic form on $X$. Somewhat abusively we denote by the same letter the extension to $H^2(X, \R)$ and by $q(.,.)$ the associated bilinear form.
The restriction of the bilinear form to $H^{1,1}(X, \R)$ has index $(1, b_2(X)-3)$ \cite[Thm.5]{Bea83}. Let us recall two well-known consequences (see e.g. \cite[Lemma 3.1]{MY15} for a proof):

\begin{lemma} \label{lemma:HI0}
Let $\alpha_1, \alpha_2 \in H^{1,1}(X, \R)$
 be $(1,1)$-classes such that
$$
0 = q(\alpha_1)=q(\alpha_2)=q(\alpha_1, \alpha_2).
$$
Then $\alpha_1$ and $\alpha_2$ are colinear in $H^{1,1}(X, \R)$.
\end{lemma}

\begin{lemma} \label{lemma:HI1}
Let $\alpha_1, \alpha_2 \in H^{1,1}(X, \R)$
be non-zero $(1,1)$-classes such that
$$
q(\alpha_1)>0, \qquad q(\alpha_2)=0
$$
Then $q(\alpha_1, \alpha_2) \neq 0$.
\end{lemma}

We also note

\begin{lemma} \label{lemma:HI2}
Let $A$ be an ample divisor on $X$.
If $D \subset X$ is an effective divisor with $q(A, D)=0$, then $D=0$.
\end{lemma}

\begin{proof}
By the index theorem we have $q(D) \leq 0$ 
and $q(D)=0$ if and only if $D=0$. Arguing by contradiction we assume $D \neq 0$ and therefore $q(D)<0$. Since $A$ is ample the class $A+\epsilon D$ is ample for
$1 \gg \epsilon>0$, so
$$
0 \leq q(A+\epsilon D, D) = q(A, D) + \epsilon q(D) = \epsilon q(D)<0,
$$
the desired contradiction.
\end{proof}

We denote by $\CX \subset H^{1,1}(X, \R)$ the positive cone, i.e. the connected component
of the cone $\{ \alpha \in H^{1,1}(X, \R) \ | \ q(\alpha)>0 \}$ that contains the K\"ahler cone of $X$.  We denote by $\BKX$ the birational K\"ahler cone and recall that
\begin{align*}
\BKX &= \{
\alpha \in \CX \ | \ q(\alpha, E) \geq 0 \ \text{ for any } \ E \subset X \mbox{ effective divisor}
\},  \\
&=\{
\alpha \in \CX \ | \ q(\alpha, E) \geq 0 \ \text{ for any } \ E \subset X \mbox{ prime uniruled divisor}
\}. 
\end{align*} 
The intersection of $\BKX$ with the Neron-Severi space $N^1(X)_\R \subset H^{1,1}(X, \R)$ is the cone of movable divisors, i.e. the closure of the cone generated by mobile divisors.

If $D$ is an effective divisor on $X$, then $q(D, E) \geq 0$ for every prime divisor $E$ not contained in the support of $D$ \cite[Prop.4.2(ii)]{Bou04}. Therefore we have
\begin{fact} \label{fact:effectiveBKX}
An effective divisor $D$ is in $\BKX$ if and only if $q(D, E) \geq 0$ for every prime
uniruled divisor $E$ contained in the support of $D$.
\end{fact}

Moreover we have
\begin{lemma} \cite[Lemma 2.13]{Rie23} \label{lemma:h-chi}
Let $D$ be a divisor on $X$ such that $D \in \BKX$ and $q(D)>0$. Then
$$
h^0(X, \sO_X(D)) = \chi(X, \sO_X(D))
$$
\end{lemma}

By \cite[Section 1.11]{Huy99} there exists a polynomial 
$P \in \Q[X]$ of degree $n$ such that for every $D \in \pic{X}$ one has
\begin{equation}
\label{eqn:RRpolynomial}
\chi(X, \sO_X(D)) = P (q(D)).
\end{equation}
Note that the constant term of the polynomial $P$ is $\chi(X, \sO_X)=n+1$ and by a recent result
of Chen Jiang \cite[Thm.1.1]{Jia23} all the other coefficients are strictly positive.
In particular 
$$
P : \Q \rightarrow \Q
$$
is a strictly increasing function. As an immediate consequence one has

\begin{corollary} \label{cor:Jiang}
Let $X$ be a hyperk\"ahler manifold of dimension $n$, and let $D \in \pic{X}$ be a divisor such that $D \in \BKX$ and $q(D)>0$. Then we have
$$
h^0(X, \sO_X(D)) = P(q(D)) > n+1.
$$
Let $D_1, D_2 \in \pic{X}$ be divisors such that $D_i \in \BKX$ and $q(D_i)>0$ for $i=1,2$. 
Then $h^0(X, \sO_X(D_1))=h^0(X, \sO_X(D_2))$ if and only if $q(D_1)=q(D_2)$.
\end{corollary}

\begin{proof}
By Lemma \ref{lemma:h-chi} we have $h^0(X, \sO_X(D)) = \chi(X, \sO_X(D))$ and by \eqref{eqn:RRpolynomial} the latter is computed by the Riemann-Roch polynomial. Since $P$ is strictly increasing one
has $P(q(D))>P(0)=n+1$. The proof of the second statement is analogous.
\end{proof}

The following consequence of Markman's work is well-known, e.g. \cite[Lemma 3.5]{Rie23}:

\begin{lemma} \label{lemma:divisibility}
Let $E \subset X$ be a prime divisor such that $q(E)<0$ and let $D$ be any divisor on $X$. Then
$$
-q(E) \ | \ 2 q(E,D).
$$
Thus if $q(E,D) \geq 0$ we can write
$$
q(E,D) = m \frac{-q(E)}{2}
$$
with $m \in \N$.
\end{lemma}

\begin{proof}
By \cite[Cor.3.6]{Mar13} the class $E^\vee := \frac{-2 q(E, \bullet)}{q(E)} \in H_2(X, \Q)$ 
is integral, i.e. a class in $H_2(X, \Z)$. Thus for every $D \in \pic{X} \subset H^2(X, \Z)$ we have
$$
\frac{-2 q(E, D)}{q(E)} = E^\vee \cdot D \in \Z.
$$
\end{proof}

\section{Base divisors on hyperk\"ahler manifolds}

We fix the setup for the whole section:

\begin{setup} \label{setup}
Let $X$ be a hyperk\"ahler manifold of dimension $2n$, and let $A \in \BKX$ be a divisor
such that $q(A)>0$ and $|A|$ has a fixed divisor.
We denote by
$$
A \simeq M+B = M + \sum_{j=1}^{k+1} b_j B_j
$$
the decomposition into mobile and fixed part where the $B_j$ are distinct prime divisors
and $b_j \in \N^*$. Since $h^0(X, \sO_X(A))>n+1$ (Cor. \ref{cor:Jiang}) it is clear that the mobile part $M$ is not zero.
Moreover $|M|$ has no fixed component, so $M \in \BKX$.
Note that for any effective divisor $B' \subset B$ we have a canonical isomorphism
$H^0(X, \sO_X(B')) \simeq H^0(X, \sO_X(B))$. In particular
\begin{equation}
\label{h0easy}
h^0(X, \sO_X(B'))=1, \qquad h^0(X, \sO_X(A))=h^0(X, \sO_X(M+B')).
\end{equation}
\end{setup}

We start with some observations that are well-known to specialists:

\begin{lemma} \label{lemma1}
In the Setup \ref{setup}, let $B' \subset B$ be an effective divisor such that $M+B' \in \BKX$ and $q(M+B')>0$. Then we have $B'=B$.
In particular we have $q(M)=0$.
\end{lemma}

\begin{proof} 
The second statement follows from the first by setting $B'=0$.

For the first statement note that 
$$
\chi(X, \sO_X(A)) = h^0(X, \sO_X(A))=h^0(X, \sO_X(M+B')) = \chi(X, \sO_X(M+B'))
$$
by \eqref{h0easy} and Lemma \ref{lemma:h-chi}. By Corollary \ref{cor:Jiang} this implies
$q(A)=q(M+B')$. 
Since $A \in \BKX$ and $M+B' \in \BKX$ we have 
\begin{multline*}
q(A)  = q(A, M+B') + q(A, B-B') \geq q(A, M+B') =
\\
= q(M+B') + q(B-B', M+B') \geq q(M+B') = q(A).
\end{multline*}
Thus all the inequalities are equalities, i.e. we have
$$
q(A, B-B') = 0, \qquad q(M+B', B-B')=0.
$$
Since $q(A)>0$ the first equality and the index theorem imply that the intersection matrix of
the effective divisor $B-B'$ is negative (or $B-B'=0$). Since $M+B' \in \BKX$ the second equality implies by \cite[Thm.4.8]{Bou04} that
$$
A \simeq (M+B') + (B-B')
$$
is the divisorial Zariski decomposition of $A$. Yet $A \in \BKX$, so the negative part $B-B'$
is zero by uniqueness of the decomposition.
\end{proof}

The intersection of the Picard lattice $\pic{X} \subset H^{1,1}(X, \R)$ with the vector space $\R M$ is a discrete subgroup of $(\R,+)$ so we have
$$
\R  M \cap \pic{X} = \Z D
$$
for some $D \in \pic{X}$. We say that $M$ is a primitive class if $D = \pm M$.

\begin{corollary} \label{cor:notprimitive}
In the Setup \ref{setup}, assume that 
$M$ is not a primitive class
in the Picard lattice $\pic{X}$. Then $B$ is a prime divisor. 
\end{corollary}

\begin{proof} Recall that $q(M)=0$ by Lemma \ref{lemma1}.
Since $M$ is not a primitive class we have $M \simeq d L$ with $d \geq 2$ and $L$ a divisor on $X$.
Recall that $A$ and $M$ are in $\BKX$. Since $q(A)>0$ and $M$ is not zero we have $q(A,M)>0$
by Lemma \ref{lemma:HI1}. Therefore
$$
0< q(A, M) = \sum b_j q(B_j, M)
$$
and up to renumbering we can assume $q(M, B_1)>0$. If $q(B_1) \geq 0$ have $B_1 \in \BKX$
and therefore $M+B_1 \in \BKX$. Since 
$$
q(M+B_1) = 2 q(M, B_1) + q(B_1) >0
$$
we conclude with Lemma \ref{lemma1} applied with $B'=B_1$.

Assume now that $q(B_1)<0$. Since 
$$
q(L, B_1) = \frac{1}{d} q(M, B_1)>0
$$ 
we know by Lemma \ref{lemma:divisibility}
that $q(L, B_1) = m \frac{-q(B_1)}{2}$ with $m \in \N^*$. Thus we have
$$
q(M+B_1, B_1)
= d q(L, B_1) + q(B_1) 
= (\frac{md}{2}-1) (-q(B_1)).
$$
Since $d \geq 2$ this shows that $q(M+B_1, B_1) \geq 0$.
Thus $M+B_1 \in \BKX$ and 
$$
q(M+B_1) = 2 q(M, B_1) + q(B_1) = (md-1) (-q(B_1)) > 0
$$
where we used again $d \geq 2$. Now conclude again with Lemma \ref{lemma1}.
\end{proof}

\begin{lemma} \label{lemma2}
In the Setup \ref{setup}, assume that 
$M$ is a primitive class
in the Picard lattice $\pic{X}$. 
Let $0 \neq B' \subset B$ be an effective divisor such that $B' \in \BKX$ and $q(B')=0$. Then we have $B'=B$.
\end{lemma}

\begin{proof}
We have $M+B' \in \BKX$, so the statement follows from
Lemma \ref{lemma1} if $q(M+B')>0$. 

Assume from now on that $q(M+B')=0$ and therefore $q(M, B')=0$.
By Lemma \ref{lemma:HI0} this implies $B' = \lambda M$ with $\lambda \in \Q^*$.
Since $\BKX$ contains no lines we have $\lambda>0$ and $M$ primitive implies that
$\lambda \in \N$. 
On a hyperk\"ahler manifold numerical equivalence coincides with linear equivalence, so we have $B' \simeq \lambda M$.
Yet then
$$
h^0(X, \sO_X(B')) \geq h^0(X, \sO_X(M)) = h^0(X, \sO_X(A))>1,
$$
a contradiction to \eqref{h0easy}.
\end{proof}

\begin{lemma} \label{lemma:technical}
In the Setup \ref{setup}, assume that there exists two prime divisors
$B_j, B_{j+1} \subset B$ such that
$$
q(B_j)<0, \qquad q(B_{j+1})<0, \qquad q(B_j, B_{j+1})=\frac{-q(B_j)}{2}.
$$
Then $q(B_j)=q(B_{j+1})$ or $-q(B_{j+1}) \leq \frac{-q(B_j)}{2}$.
\end{lemma}

\begin{proof}
Applying Lemma \ref{lemma:divisibility} to $E=B_{j+1}$ we can write
$$
\frac{-q(B_j)}{2} =  q(B_j, B_{j+1}) = m \frac{-q(B_{j+1})}{2}
$$
with $m \in \N^*$. In particular $-q(B_{j+1})$ divides $-q(B_j)$. Thus
if $-q(B_j) \neq -q(B_{j+1})$ we have  $-q(B_{j+1}) \leq \frac{-q(B_j)}{2}$.
\end{proof}

\begin{lemma} \label{lemma:key}
In the Setup \ref{setup}, let $B_j \subset B$ be a prime divisor such that
$q(M, B_j)>0$. Then one of the following holds
\begin{itemize}
\item $M+B_j \in \BKX$
\item $M+B_j \not\in \BKX$ and $2M+B_j \in \BKX$.  
\end{itemize}
In the second case we have 
$$
q(M,B_j) = \frac{-q(B_j)}{2}, \qquad 
q(2M+B_j) = - q(B_j) >0, \qquad
\mbox{and} \qquad 
q(A)< - q(B_j).
$$
\end{lemma}

\begin{proof} If $q(B_j) \geq 0$ it is clear that $M+B_j \in \BKX$, so we can 
assume that $q(B_j)<0$. 
By Lemma \ref{lemma:divisibility} applied with $E=B_j$ we have
$$
q(M, B_j) = m \frac{-q(B_j)}{2}
$$
with $m \in \N^*$. Thus 
$$
q(M+B_j, B_j) = (m-2) \frac{-q(B_j)}{2} < 0
$$
if and only if $m=1$ and in this case $q(2M+B_j,B_j)=0$. Thus 
 $M+B_j \not\in \BKX$ implies $m=1$ and $2M+B_j \in \BKX$.   
 
Suppose now that we are in the second case: since $m=1$ the computation above shows
$q(2M+B_j) = -q(B_j)>0$.
Arguing by contradiction we assume that $q(A) \geq -q(B_j)$, then Corollary \ref{cor:Jiang} yields
$$
h^0(X, \sO_X(A)) \geq h^0(X, \sO_X(2M+B_j)) \geq h^0(X, \sO_X(2M)).
$$
Since $h^0(X, \sO_X(M)) \geq 2$ we have $h^0(X, \sO_X(2M))>h^0(X, \sO_X(M))$ by Lemma \ref{lemma:strict-increase}. Thus
$$
h^0(X, \sO_X(A))  \geq h^0(X, \sO_X(2M)) > h^0(X, \sO_X(M)) = h^0(X, \sO_X(A))
$$
since $M$ is the mobile part of $|A|$, a contradiction.
\end{proof}

\begin{theorem} \label{theorem-technical}
In the situation of Setup \ref{setup} one of the following holds
\begin{itemize}
\item $B=B_1$ is a prime divisor with $q(B_1) \leq 0$; or
\item $B= \sum_{j=1}^{k+1} B_j$ is a reduced divisor and the dual graph of $M+ \sum_{j=1}^{k+1} B_j$ is 
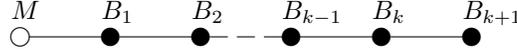
\begin{figure}[!h]
\setlength{\unitlength}{1.2mm}
\centering
\begin{picture}(120,10)
\put(25,1){\circle{2}}
\multiput(35,1)(10,0){4}{\circle*{2}}
\put(75,1){\circle*{2}}
\multiput(26,1)(10,0){2}{\line(1,0){8}}
\multiput(46,1)(3,0){3}{\line(3,0){2}}
\multiput(56,1)(10,0){2}{\line(1,0){8}}
\put(24,3){$M$}
\put(34,3){$B_1$}
\put(44,3){$B_2$}
\put(54,3){$B_{k-1}$}
\put(64,3){$B_k$}
\put(74,3){$B_{k+1}$}
\end{picture}
\end{figure}

We have $q(B_j)=q(B_1)<-1$ for all $j \in \{1, \ldots, k\}$ and $-1 \geq q(B_{k+1}) \geq \frac{-q(B_1)}{2}$. Moreover 
$$
q(M, B_1) = q(B_j, B_{j+1}) = \frac{-q(B_1)}{2}, \quad q(A,B_j) = 0
$$
for all $j \in \{1, \ldots, k\}$.
\end{itemize}
In both cases we have $q(M)=0, q(B) \leq 0$ and $q(M, B)>0$.
\end{theorem}

\begin{proof}[Proof of Theorem \ref{theorem-technical}]
We recall the notation
$$
B = \sum_{j=1}^{k+1} b_j B_j
$$
from the Setup \ref{setup}. We assume that $B \neq B_1$, our goal is to obtain the description of $B$ from the second case.

Recall that $M \in \BKX$ and $q(M)=0$ by Lemma \ref{lemma1}. By Corollary \ref{cor:notprimitive} we can assume without loss of generality that $M$ is a primitive class in the Picard lattice.
By Lemma \ref{lemma2} we can assume $q(B_j)<0$ for all $j=1, \ldots, k+1$. 

Since $0< q(A, M) = \sum b_j q(B_j, M)$ there exists a divisor $B_j$ such that $q(M,B_j)>0$.
Among these divisor we choose $B_j$ such that $0<-q(B_j)$ is minimal.
We can assume up to renumbering that $j=1$. 

{\em Step 1. We show that $b_1=1$ and $q(B_1)<-1$. Moreover $q(A, B_1)=0$ and $q(M, B_j)=0$ for all $j \geq 2$.}

If $M+B_1 \in \BKX$ we have $q(M+B_1) \geq q(M, B_1)>0$ and therefore $A=M+B_1$ by Lemma \ref{lemma1}. Thus $M+B_1 \not\in \BKX$, so  Lemma \ref{lemma:key} yields
$q(A)<-q(B_1)$  and $q(M,B_1)=\frac{-q(B_1)}{2}$. 
Since $q(A) \in \N^*$ this shows $q(B_1)<-1$. 

Arguing by contradiction we assume that (up to renumbering $B_2, \ldots, B_{k+1}$) we have
$q(M, B_2)>0$. By our minimality assumption and Lemma \ref{lemma:divisibility} (applied with $E=B_2$) this implies
\begin{multline*}
q(A) = q(A, M) + q(A, B) \geq q(A, M) = q(B, M) \\
= \sum b_j q(B_j, M) \geq \frac{-q(B_1)}{2} + 
\frac{-q(B_2)}{2} \geq -q(B_1),
\end{multline*}
a contradiction. Thus we have $q(M, B_j)=0$ for all $j \geq 2$. 
In the same spirit
\begin{multline*}
q(A) = q(A, M) + q(A, B-B_1) + q(A, B_1) \geq q(A, M) + q(A,B_1) = \\
= q(B, M) + q(A,B_1) = b_1 q(B_1, M) + q(A,B_1) 
\end{multline*}
is greater or equal than $-q(B_1)$ unless $b_1=1$ and $q(A,B_1)=0$
(here we use again Lemma \ref{lemma:divisibility} with $E=B_1$).

In order to simplify the notation we set $d:=q(B_1)$.

{\em Step 2. We show that (up to renumbering) $B_2, \ldots, B_{k+1}$ we have $b_2=1$,
$q(B_1, B_2) = \frac{-d}{2}$ and $q(B_1, B_j)=0$ for $j > 2$.}
Set $B':=B-B_1$, then $B_1$ is not contained in the support of $B'$ by Step 1. Moreover
$$
0 = q(A, B_1) = q(M,B_1) + q(B_1) + \sum_{j=2}^{k+1} b_j q(B_j, B_1) 
= \frac{d}{2} + \sum_{j=2}^{k+1} b_j q(B_j, B_1) 
$$
By Lemma \ref{lemma:divisibility} (applied with $E=B_1$) all the numbers $q(B_j, B_1)$ are non-negative multiples of
$\frac{-d}{2}$ so only one of them is positive and the corresponding multiplicity is one.
This finishes the proof of Step 2 and establishes all the intersection numbers from
the theorem involving $B_1$.

{\em Step 3. We show by induction: for all $l=1, \ldots, k$ we have $q(B_l)=d$ and $q(A, B_l)=0$. Moreover (up to renumbering $B_{l+1}, \ldots, B_{k+1}$), we have 
$$
b_{l+1}=1, \qquad
q(B_l, B_{l+1}) = \frac{-d}{2},  \qquad \mbox{and} \qquad q(B_l, B_j)=0 \qquad \forall j > l+1.
$$
}
The case $l=1$ has been settled by the Steps 1 and 2, let us do the induction step $l-1 \Rightarrow l$. If $q(B_l) \geq \frac{d}{2}$ we have
$$
q(M+ \sum_{j=1}^l B_j, B_l) \geq \frac{-d}{2} + \frac{d}{2} \geq 0.
$$
By the induction hypothesis $q(M+ \sum_{i=1}^l B_i, B_j)=0$ for $j=1, \ldots,l-1$, so  $M+ \sum_{j=1}^l B_j \in \BKX$.
Moreover
$$
q(M+ \sum_{j=1}^l B_j) \geq q(M+ \sum_{j=1}^l B_j, M) = \frac{-d}{2}>0, 
$$
so $B = \sum_{j=1}^l B_j$ by Lemma \ref{lemma1}. Yet this is a contradiction since $l < k+1$. Thus we have $q(B_l) < \frac{d}{2}$
and therefore $q(B_l)=q(B_{l-1})=d$ by Lemma \ref{lemma:technical} and the induction hypothesis.
Recall that $-d > q(A)$ by Lemma \ref{lemma:key} so
$$
-d > q(A) = q(A, M) + q(A,B) \geq \frac{-d}{2} + q(A, B_l).
$$
Since $q(A, B_l)$ is a non-negative multiple of $\frac{-q(B_l)}{2}=\frac{-d}{2}$ by Lemma \ref{lemma:divisibility} we obtain $q(A, B_l)=0$.

Set $B^{(l)}=B-\sum_{j=1}^l B_j$ then the support of $B^{(l)}$ is the union of
$B_{l+1}, \ldots, B_{k+1}$. We have
\begin{multline*}
0 = q(A, B_l) = q(M+\sum_{j=1}^{l-1} B_j, B_l) + q(B_l) + \sum_{j=l+1}^{k+1} b_j q(B_j, B_l) =
\\
\frac{-d}{2} + d + \sum_{j=l+1}^{k+1} b_j q(B_j, B_l).
\end{multline*}
By Lemma \ref{lemma:divisibility} applied with $E=B_l$ the numbers $q(B_j, B_l)$ are non-negative multiples
of $\frac{-q(B_l)}{2}=\frac{-d}{2}$ so only one of them is positive and the corresponding multiplicity is one.

{\em Step 4. Conclusion.} By the Steps 1, 2, 3 we have
$b_j=1$ for all $j=1, \ldots, k+1$, moreover we have determined all the intersection
numbers $q(B_i,B_j)$ except $q(B_{k+1})$. The divisor $A$ is in $\BKX$, so 
$$
0 \leq q(A, B_{k+1}) = q(M+\sum_{j=1}^{k} B_j, B_{k+1}) + q(B_{k+1})
= \frac{-d}{2} + q(B_{k+1}).
$$
Thus we have $-q(B_{k+1}) \leq \frac{-d}{2}$, in particular $q(B_{k+1})$ is distinct
from $q(B_k)$. 

If $B$ is not a prime divisor the computation above implies $q(B)=q(B_{k+1})<0$.
If $B=B_1$ is a prime divisor we still have $q(B) \leq 0$ since otherwise
$B \in \BKX$ and $h^0(X, B) \geq n+1$ by Corollary \ref{cor:Jiang}.
Thus $q(B) \leq 0$ holds in both cases, the property $q(M,B)>0$ is now a consequence of $q(A)>0, q(M)=0$.
\end{proof}

\begin{proof}[Proof of Theorem \ref{theorem:main}]
The nef and big divisor $A$ is in $\BKX$ and $q(A)>0$, so $h^0(X, A) \geq 2$ by Corollary 
\ref{cor:Jiang}.
By Theorem \ref{theorem-technical} the fixed part is reduced, so the first statement is clear.

For the second statement we follow the argument in 
\cite[Cor.4.9]{Rie23}:
let $A \simeq M+B$ be the decomposition into mobile and fixed part, by assumption we have $B \neq 0$.
Consider now the linear system $|2A|$ and let 
$2A \simeq M' + B'$
be the decomposition into mobile and fixed part. Since 
$$
2M + 2B \simeq 2 A \simeq M' + B'
$$
and $2M$ is mobile we have $B' \subset 2B$. By Theorem \ref{theorem-technical}
the fixed part $B'$ is reduced, so we even have $B' \subset B$. In conclusion $M' \simeq 2M+B+(B-B')$ and $B-B'$ is an effective divisor.
Since $M'$ and $M$ are in $\BKX$ we have
$$
q(M') \geq q(M', 2M+B) \geq 2 q(M',M) \geq 2 q(2M+B, M) = 2 q(B,M)>0
$$
where the strict inequality is again due to Theorem \ref{theorem-technical}.
By Lemma \ref{lemma1} this implies $B'=0$, i.e. $|2A|$ has no fixed divisors.

If $A$ is ample we have $q(A, B_1)>0$ by Lemma \ref{lemma:HI2}, this excludes the second
case in Theorem \ref{theorem-technical}.  This shows the last statement.
\end{proof}

\begin{remark*}
The proof of the first two statements only uses that $A \in \BKX$ and $q(A)>0$. For the third statement we need the additional assumption that $q(A,E)>0$ for every prime
divisor $E$ with $q(E)<0$.
\end{remark*}

\section{Proof of Theorem \ref{thm:HK4}}

\begin{proof}
Let $A \simeq M + B$ be the decomposition into mobile and fixed part. Since $M \in \BKX$ there
exists an MMP  
$$
\holom{\tau}{X}{X'}
$$
to a hyperk\"ahler manifold $X'$ such that $M':=\tau_* M$ is nef \cite[Thm.1.2]{MZ13}. Any birational map between hyperk\"ahler manifolds is an isomorphism in codimension one, so setting
$A':=\tau_* A$ and $B' := \tau_* B$ 
we have
$H^0(X, \sO_X(A)) \simeq H^0(X', \sO_{X'}(A'))$ and
$$
A' \simeq M' + B'
$$
is the decomposition into mobile and fixed part of $A'$. In particular
$h^0(X', \sO_{X'}(M'))=h^0(X', \sO_{X'}(A')) \geq 4$, so $\kappa(X', M') \geq 1$. Since $M'$ is nef we can apply
\cite[Thm.1.5]{Fuk02} to the pair $(X', \epsilon M')$ for some $0<\epsilon \ll 1$ to obtain that $M'$ is semiample. Since $q(M')=q(M)=0$ by Lemma \ref{lemma1} we obtain that some positive multiple of $M'$ defines a Lagrangian fibration
$$
\holom{f}{X'}{S'}
$$
onto a surface $S'$. By \cite{Ou19,HX22} (cf. \cite{MX25} for a short proof) we have $S' \simeq \PP^2$. We claim that $M' \simeq f^* d H$ with $d \geq 2$ and $H$ the hyperplane class on $\PP^2$:
let $D \in |M'|$ be a general element. Since $M' \sim_\Q \lambda f^* H$ for some $\lambda \in \Q$
the divisor $D$ has no $f$-horizontal component, 
so $f(D) \subset \PP^2$ is an effective divisor and $D \subset f^* f(D)$.
The linear system $|M'|$ has no fixed components, so none of the irreducible components
of $D$ is contained in $\fibre{f}{\Delta}$ where $\Delta \subset \PP^2$ is the discriminant locus. Therefore
we have $D = f^* f(D)$, in particular $M' \simeq f^* d H$ for some $d \in \N$.
Since $h^0(X', \sO_{X'}(M'))>3=h^0(\PP^2, \sO_{\PP^2}(H))$ it is clear that $d \geq 2$.

Set $L:=\tau^{-1}_* f^* H$, then $M \simeq dL$, so $M$ is divisible in the Picard lattice and therefore
the fixed part $B$ is a prime divisor $B_1$ by Corollary \ref{cor:notprimitive}. From the construction it is clear that $L$ has the properties claimed in Theorem \ref{thm:HK4}.   

Finally let us show that $q(B_1)<0$: if $q(B_1)=0$ the divisor $L+B_1$ is in $\BKX$ and
$q(L+B_1)=q(L,B_1)=\frac{1}{d}q(M,B_1)>0$ by Theorem \ref{theorem-technical}. Thus $h^0(X, \sO_X(L+B_1))>3$ by Corollary \ref{cor:Jiang}. Yet $B_1$ is a fixed component of $dL+B_1$, so
it is a fixed component of $L+B_1$. Thus we have
$$
h^0(X, \sO_X(L+B_1)) = h^0(X, \sO_X(L)) = 3,
$$
the final contradiction.
\end{proof}


\providecommand{\bysame}{\leavevmode\hbox to3em{\hrulefill}\thinspace}
\providecommand{\MR}{\relax\ifhmode\unskip\space\fi MR }
\providecommand{\MRhref}[2]{%
  \href{http://www.ams.org/mathscinet-getitem?mr=#1}{#2}
}
\providecommand{\href}[2]{#2}

\end{document}